\documentclass[11pt,oneside,english]{amsart}
\usepackage{newcent}
\usepackage[T1]{fontenc}
\usepackage[latin9]{inputenc}
\usepackage{geometry}
\usepackage{amstext}
\usepackage{amsthm}
\usepackage{amssymb}

\makeatletter
\numberwithin{equation}{section}
\numberwithin{figure}{section}
\theoremstyle{definition}
\newtheorem{defn}{\protect\definitionname}[section]
\theoremstyle{plain}
\newtheorem{prop}{\protect\propositionname}[section]
\theoremstyle{remark}
\newtheorem{rem}{\protect\remarkname}[section]
\theoremstyle{plain}
\newtheorem{lem}{\protect\lemmaname}[section]
\theoremstyle{plain}
\newtheorem{thm}{\protect\theoremname}[section]
\theoremstyle{definition}
\newtheorem{example}{\protect\examplename}[section]

\makeatother

\usepackage{babel}
\providecommand{\definitionname}{Definition}
\providecommand{\examplename}{Example}
\providecommand{\lemmaname}{Lemma}
\providecommand{\propositionname}{Proposition}
\providecommand{\remarkname}{Remark}
\providecommand{\theoremname}{Theorem}

\begin{document}
\title{orthogonal coordinates on 4 dimensional Kähler manifolds }
\author{David L. Johnson}
\curraddr{Department of Mathematics, Lehigh University, 17 Memorial Walk, Bethlehem,
PA  18015, USA}
\email{dlj0@lehigh.edu}
\thanks{The author wishes to thank Vincent Borrelli for many helpful comments
and suggestions, and also Andrei Moroianu for suggesting a much more
elegant approach to describing the complex structure tensor.}
\date{\today}
\begin{abstract}
The existence of orthogonal local coordinates is a generalization
of the manifold being conformally flat. It is always possible to construct
orthogonal coordinates on 2-manifolds, using geometric normal coordinates
or isothermal coordinates. In 1984, Dennis DeTurck and Dean Yang \cite{DeTurck-Yang}
showed the existence of orthogonal coordinates on any Riemannian 3-manifold.
Thus there are manifolds which have orthogonal coordinates, but are
not conformally flat, since the Cotton tensor presents an obstruction
to conformal flatness in dimension 3. They also showed that, for dimensions
at least 4, there is apparently an obstruction to the existence of
orthogonal coordinates, in that curvature components of the form $R_{ijkl}$,
with all 4 indices distinct, will vanish if the directions correspond
to orthogonal coordinates. Thus, in high dimensions, the existence
of orthogonal coordinates implies a certain sparseness of the Riemannian
curvature tensor. Recently, Paul Gauduchon and Andrei Moroianu showed
\cite{GM} that there are no orthogonal coordinates on $\mathbb{CP}^{n}$
or $\mathbb{HP}^{n}$, if $n>1$. 

The main results of this work are that no nontrivial self-dual Kähler
4-manifold (4 real dimensions) supports orthogonal local coordinates,
and also no nontrivial Ricci-flat Kähler 4-manifold supports orthogonal
coordinates. The first result uses the same technique developed by
Gauduchon and Moroianu in the special case of $\mathbb{CP}^{2}$ with
the Fubini-Study metric, but the second result uses purely algebraic
methods. 
\end{abstract}

\maketitle

\section{Introduction}

A local coordinate chart $\left\{ x_{1},\ldots,x_{n}\right\} $ on
a Riemannian manifold $M$ is \emph{orthogonal} if $\left<\frac{\partial}{\partial x_{i}},\frac{\partial}{\partial x_{j}}\right>=0$
whenever $i\neq j$. In that case, the Riemannian metric can be written
as $g=a_{1}^{2}dx_{1}^{2}+\cdots+a_{n}^{2}dx_{n}^{2}$, with of course
$a_{i}=\sqrt{\left<\frac{\partial}{\partial x_{i}},\frac{\partial}{\partial x_{i}}\right>}$.
In dimension 2, the existence of orthogonal coordinates is classical,
going back to the construction of geodesic normal coordinates along
a curve, and to isothermal coordinates \cite{Chern}. DeTurck and
Yang \cite{DeTurck-Yang} show that any 3-manifold has systems of
orthogonal coordinates, even those which are not conformally flat. 

In dimensions larger than 3, the existence of orthogonal coordinates
implies that, in the directions of those coordinates, $R_{ijkl}=0$
whenever all indices are distinct \cite{DeTurck-Yang}, so that the
curvature operator would be somewhat sparse in high dimensions. 

Particularly on Kähler manifolds, the existence of orthogonal coordinates
seems to be extremely restrictive, since the Kähler structure tends
to obstruct differential-geometric restrictions \cite{kaehler submersions}.
The results of Gauduchon and Moroianu \cite{GM} show that, in particular,
$\mathbb{CP}^{n}$, $n\geq2$, with the Fubini-Study metric does not
admit orthogonal coordinates. It is easy to show that there are no
orthogonal \emph{holomorphic} coordinates (below), but the harder
issue would be to show that no \emph{real }orthogonal coordinates
could be found, which Gauduchon and Moroianu did indeed establish.
They also show that the trivial examples of local products of Riemann
surfaces are not the only Kähler manifolds which do support orthogonal
coordinates.

The existence of orthogonal coordinates is a generalization of the
manifold being conformally flat, since in that case there is are local
coordinate charts so that the metric is of the form $g_{ij}=a\delta_{ij}$,
or, the coordinates are orthogonal with $a_{i}=a_{j}$ for all $i,j$.
Conformally-flat Kähler manifolds have been classified by \cite{Tanno,Tachi},
which in dimension 4 must either be flat or a product of a Riemann
surface of constant curvature, and one of the opposite constant curvature
\cite{Tanno}. In dimensions at least 6, Yano and Mogi \cite{Yano}
showed that a conformally-flat Kähler manifold must be flat.

In real dimension 4, using an elegant decomposition of the curvature
tensor due to Atiyah-Hitchen-Singer, Singer-Thorpe, and Claude LeBrun
\cite{AHS,Singer-Thorpe,LeBrun}, we show a slight extension of Gauduchon
and Moroianu's result to show that no self-dual, 4-dimensional Kähler
manifold supports orthogonal coordinates. The method of proof of this
result is due to Gauduchon and Moroianu \cite{GM}.

We are also able to establish that no nonflat Kähler 4-manifold which
is Ricci-flat supports orthogonal coordinates. Kähler 4-manifolds
which are Ricci-flat are K3 surfaces, 4-dimensonal Calabi-Yau manifolds.

Throughout this work we refer to the dimension of a manifold as its
real dimension, even if it is a complex manifold. We refer to a local
orthonormal moving frame $\left\{ e_{1},\ldots,e_{n}\right\} $ as
a \emph{frame} of (a neighborhood of) the space. In the case of a
complex manifold $M$ with complex-structure tensor $J:T_{*}\left(M\right)\to T_{*}\left(M\right)$,
we call a frame $\left\{ e_{1},\ldots,e_{2m}\right\} $ \emph{unitary}
if $Je_{2k+1}=e_{2k+2}$ and $Je_{2k+2}=-e_{2k+1}$ $k\in\left\{ 0,\ldots,m-1\right\} $.

\section{Orthogonal coordinates}
\begin{defn}
A Riemannian manifold $M^{n}$ with Riemannian metric $\left<\,,\,\right>$
has \emph{orthogonal coordinates} $\left\{ x_{1},\ldots,x_{n}\right\} $
in a neighborhood $U$ if, for each point $x\in U$, $\left<\frac{\partial}{\partial x_{i}},\frac{\partial}{\partial x_{j}}\right>=0$
whenever $i\neq j$. If each point of $M$ has orthogonal coordinates
in some neighborhood, then we say that $M$ \emph{has orthogonal coordinates
}or \emph{supports orthogonal coordinates}. We set (following the
notation of \cite{GM}), $a_{i}=\sqrt{\left<\frac{\partial}{\partial x_{i}},\frac{\partial}{\partial x_{i}}\right>}>0$. 
\end{defn}
Given an orthogonal coordinate system, the \emph{associated} frame
$\left\{ e_{1},\ldots,e_{n}\right\} $ is defined by setting $e_{i}:=\frac{1}{a_{i}}\frac{\partial}{\partial x_{i}}$.
From the definitions of the covariant derivative and Riemann curvature
tensor, Gauduchon and Moroianu show \cite{GM} the following useful
results, using the conventions that $R_{XY}Z:=\nabla_{\left[X,Y\right]}Z-\nabla_{X}\left(\nabla_{Y}Z\right)+\nabla_{Y}\left(\nabla_{X}Z\right)$,
and (with respect to the associated frame), $R_{ijkl}=\left<R_{e_{i}e_{j}}e_{k},e_{l}\right>$.
The notation here is somewhat different from \cite{GM}, but is equivalent.
\begin{prop}
\label{GM-proposition}\textbf{\emph{{[}Proposition 2.5 of \cite{GM}{]}}}
Let $M$ be a Riemannian manifold with orthogonal coordinates on a
chart $U$, with associated frame $\left\{ e_{1},\ldots,e_{n}\right\} $.
Then, 
\begin{enumerate}
\item If $i\neq j$, then $\nabla_{e_{i}}e_{j}=\frac{1}{a_{i}}e_{j}\left(a_{i}\right)e_{i}$,
so $\left[e_{i},e_{j}\right]=\frac{1}{a_{i}}e_{j}\left(a_{i}\right)e_{i}-\frac{1}{a_{j}}e_{i}\left(a_{j}\right)e_{j}$.
\item $\nabla_{e_{i}}e_{i}=-\sum_{j\neq i}\frac{1}{a_{i}}e_{j}\left(a_{i}\right)e_{j}$.
\item For any $i\neq j$, 
\begin{eqnarray*}
R_{e_{i}e_{j}}e_{i} & = & -\frac{1}{a_{j}}e_{i}e_{i}\left(a_{j}\right)e_{j}-\frac{1}{a_{i}}e_{j}e_{j}\left(a_{i}\right)e_{j}-\sum_{l\neq i,j}\left(\frac{1}{a_{i}}e_{j}e_{l}\left(a_{i}\right)e_{l}\right)\\
 &  & +\sum_{l\neq i,j}\left(\frac{1}{a_{i}}e_{j}\left(a_{i}\right)\frac{1}{a_{j}}e_{l}\left(a_{j}\right)e_{l}\right)-\sum_{l\neq i,j}\left(\frac{1}{a_{i}}e_{l}\left(a_{i}\right)\frac{1}{a_{j}}e_{l}\left(a_{j}\right)e_{j}\right).
\end{eqnarray*}
\item If $i,\,j,\,k$ are all distinct, 
\begin{eqnarray*}
R_{e_{i}e_{j}}e_{k} & = & \frac{1}{a_{i}}\nabla_{e_{j}}\left(da_{i}\right)\left(e_{k}\right)e_{i}-\frac{1}{a_{j}}\nabla_{e_{i}}\left(da_{j}\right)\left(e_{k}\right)e_{j},
\end{eqnarray*}
In particular, if $i,\,j,\,k,\,l$ are all distinct, $R_{ijkl}:=\left<R_{e_{i}e_{j}}e_{k},e_{l}\right>=0$.
\end{enumerate}
\end{prop}
\begin{rem}
Statement (4) of this result was originally shown by DeTurck and Yang
\cite{DeTurck-Yang}. 
\end{rem}

\section{Dimension 4}

Assume now that the (real) dimension of $M$ is 4. In this case, the
curvature operator has a particularly nice decomposition. Following
\cite{Singer-Thorpe}, the curvature of any Riemannian manifold decomposes
into invariant components. The space of all such \emph{algebraic curvature
operators} in dimension $n$ is the space $\mathcal{R}_{n}:=\Lambda_{2}\left(\mathbb{R}^{n}\right)\circ\Lambda_{2}\left(\mathbb{R}^{n}\right)$
of symmetric operators on $\Lambda_{2}\left(\mathbb{R}^{n}\right)$,
with the standard inner product on $\Lambda_{2}\left(\mathbb{R}^{n}\right)$,
and $\left<R_{\left(v_{1},v_{2}\right)}v_{3},v_{4}\right>=\left<R\left(v_{1}\wedge v_{2}\right),v_{3}\wedge v_{4}\right>$.

$SO(n)$ operates on this space induced from its natural action on
$\mathbb{R}^{n}$, which, following Weyl \cite{Weyl}, decomposes
$\mathcal{R}_{n}$ into an orthogonal direct sum of invariant subspaces,
which is shown in \cite{Singer-Thorpe} to be
\begin{eqnarray*}
\mathcal{R}_{n} & = & \mathcal{I}\oplus\mathcal{RIC}_{0}\oplus\mathcal{W}\oplus\mathcal{S},
\end{eqnarray*}
where $\mathcal{I}$ are all multiples of the identity operator, $\mathcal{RIC}_{0}$
corresponds to the trace-free portion of the Ricci tensor, $\mathcal{W}$
is the Weyl tensor component, and $\mathcal{S}$ are those tensors
orthogonal to the kernel of the Bianchi map $b:\mathcal{R}_{n}\to\mathcal{R}_{n}$
defined by $b\left(R\right)_{\left(v_{1},v_{2}\right)}v_{3}=R_{\left(v_{1},v_{2}\right)}v_{3}+R_{\left(v_{2},v_{3}\right)}v_{1}+R_{\left(v_{3},v_{1}\right)}v_{2}$,
so any operator realizable as the Riemann curvature tensor of a manifold
satisfies $b\left(R\right)=0$. $\mathcal{S}$ is easily seen to be
isomorphic to $\Lambda_{4}\left(\mathbb{R}^{n}\right)$. 

Denote by $\rho:\mathcal{R}_{n}\to\mathbb{R}^{n}\circ\mathbb{R}^{n}$
the Ricci contraction $\left<\rho\left(R\right)v,w\right>=\sum_{i=1}^{n}\left<R\left(v\wedge e_{i}\right),w\wedge e_{i}\right>$,
for any orthonormal frame $\left\{ e_{1},\ldots,e_{n}\right\} $.
This contraction is an equivariant map under the natural actions of
$SO(n)$, thus its kernel, and its orthogonal complement, are both
invariant subspaces, with $\mathcal{I}\oplus\mathcal{RIC}_{0}$ being
the orthogonal complement of the kernel of $\rho$. Singer and Thorpe
also define a mapping $s:\mathbb{R}^{n}\circ\mathbb{R}^{n}\to\mathcal{R}_{n}$,
which is a right inverse of $\rho$, by 
\begin{eqnarray*}
s\left(T\right)_{\left(u_{1},u_{2}\right)}\left(u_{3}\right) & = & \frac{1}{n-2}\left(\left<u_{1},u_{3}\right>T\left(u_{2}\right)-\left<u_{2},u_{3}\right>T\left(u_{1}\right)+\left<T\left(u_{1}\right),u_{3}\right>u_{2}-\left<T\left(u_{2}\right),u_{3}\right>u_{1}\right)\\
 &  & -\frac{1}{\left(n-2\right)\left(n-1\right)}{\rm tr}\left(T\right)\left(\left<u_{1},u_{3}\right>u_{2}-\left<u_{2},u_{3}\right>u_{1}\right),
\end{eqnarray*}
or, more simply, given a frame $\left\{ e_{1},\ldots,e_{n}\right\} $
of eigenvectors of $T$, with corresponding eigenvalues $\left\{ \lambda_{i}\right\} $,
as an operator, 
\begin{eqnarray}
s\left(T\right)\left(e_{i}\wedge e_{j}\right) & = & \frac{1}{n-2}\left(\lambda_{i}+\lambda_{j}-\frac{1}{n-1}{\rm tr}\left(T\right)\right)e_{i}\wedge e_{j}.\label{eq:s(T)}
\end{eqnarray}
In dimension 4 this decomposition has a particularly simple form,
due to \cite{AHS,LeBrun}. In that dimension the Hodge star operator
$\star:\Lambda_{2}\left(\mathbb{R}^{4}\right)\to\Lambda_{2}\left(\mathbb{R}^{4}\right)$
is a curvature operator; in particular it is a basis of the space
$\mathcal{S}$ of operators orthogonal to $\ker\left(b\right)$. The
operator $\star$ is defined on any oriented orthonormal frame $\left\{ e_{1},e_{2},e_{3},e_{4}\right\} $
by $\star\left(e_{1}\wedge e_{2}\right)=e_{3}\wedge e_{4}$, $\star\left(e_{1}\wedge e_{3}\right)=-e_{2}\wedge e_{4}$,
and $\star\left(e_{1}\wedge e_{4}\right)=e_{2}\wedge e_{3}$. Since
$\star^{2}=Id$, it decomposes $\Lambda_{2}\left(\mathbb{R}^{4}\right)$
into two 3-dimensional subspaces $\Lambda_{2}^{+}\left(\mathbb{R}^{4}\right)$
and $\Lambda_{2}^{-}\left(\mathbb{R}^{4}\right)$ consisting of the
$\pm1$-eigenspaces of $\star$, 
\[
\Lambda_{2}^{+}\left(\mathbb{R}^{4}\right)={\rm Span}\left\{ e_{1}\wedge e_{2}+e_{3}\wedge e_{4},e_{1}\wedge e_{3}-e_{2}\wedge e_{4},e_{1}\wedge e_{4}+e_{2}\wedge e_{3}\right\} 
\]
 and 
\[
\Lambda_{2}^{-}\left(\mathbb{R}^{4}\right)={\rm Span}\left\{ e_{1}\wedge e_{2}-e_{3}\wedge e_{4},e_{1}\wedge e_{3}+e_{2}\wedge e_{4},e_{1}\wedge e_{4}-e_{2}\wedge e_{3}\right\} .
\]
 These spaces switch under a change of orientation. 
\begin{rem}
There is a similar decomposition for Kähler curvature operators \cite{Johnson},
but here we will continue to use the real decomposition, even for
Kähler manifolds. The Ricci map and Bianchi map remain the same in
\cite{Johnson} (the Ricci tensor of a Kähler manifold will of course
be complex-linear).
\end{rem}
\begin{lem}
If $T\in\mathbb{R}^{4}\circ\mathbb{R}^{4}$ has ${\rm tr}\left(T\right)=0$,
then $s\left(T\right):\Lambda_{2}^{+}\left(\mathbb{R}^{4}\right)\to\Lambda_{2}^{-}\left(\mathbb{R}^{4}\right)$.
\end{lem}
\begin{proof}
This follows immediately from equation (\ref{eq:s(T)}).
\end{proof}
If $W\in\mathcal{W}$, then $W$ decomposes as $W=W^{+}+W^{-}$, where
$W^{+}:\Lambda_{2}\left(\mathbb{R}^{4}\right)\to\Lambda_{2}^{+}\left(\mathbb{R}^{4}\right)$
and $W^{-}:\Lambda_{2}\left(\mathbb{R}^{4}\right)\to\Lambda_{2}^{-}\left(\mathbb{R}^{4}\right)$
are the orthogonal projections of $W$ onto the indicated subspaces.
Since $W^{\pm}\perp s\left(T\right)$ for any trace-free $T$, $W^{+}:\Lambda_{2}^{+}\left(\mathbb{R}^{4}\right)\to\Lambda_{2}^{+}\left(\mathbb{R}^{4}\right)$
and $W^{-}:\Lambda_{2}^{-}\left(\mathbb{R}^{4}\right)\to\Lambda_{2}^{-}\left(\mathbb{R}^{4}\right)$.
Again, the spaces $W^{\pm}$ switch under a change of orientation. 

Given any frame $\left\{ e_{1},e_{2},e_{3},e_{4}\right\} $, we will
use the \emph{adapted frame}, for any orthonormal frame $\left\{ e_{1},e_{2},e_{3},e_{4}\right\} $
of $\Lambda_{2}\left(\mathbb{R}^{4}\right)$, as
\begin{align*}
 & \left\{ \frac{1}{\sqrt{2}}\left(e_{1}\wedge e_{2}+e_{3}\wedge e_{4}\right),\frac{1}{\sqrt{2}}\left(e_{1}\wedge e_{3}-e_{2}\wedge e_{4}\right),\frac{1}{\sqrt{2}}\left(e_{1}\wedge e_{4}+e_{2}\wedge e_{3}\right),\right.\\
 & ,\sum_{j}a_{ij}^{2}=1\left.\frac{1}{\sqrt{2}}\left(e_{1}\wedge e_{2}-e_{3}\wedge e_{4}\right),\frac{1}{\sqrt{2}}\left(e_{1}\wedge e_{3}+e_{2}\wedge e_{4}\right),\frac{1}{\sqrt{2}}\left(e_{1}\wedge e_{4}-e_{2}\wedge e_{3}\right)\right\} .
\end{align*}
 Then, we have, 
\begin{prop}
\textbf{\emph{\label{=00005BSinger-Thorpe,-AHS=00005D}\cite{AHS,Singer-Thorpe,LeBrun}}}
Any $R\in\mathcal{R}_{4}$ satisfying the Bianchi identity decomposes
into block form, 
\begin{eqnarray*}
R & = & \left[\begin{array}{c|c}
W^{+}\left(R\right)+\frac{r}{12}Id & s\left(\rho\left(R\right)-\frac{r}{4}I\right)^{\pm}\\
\hline s\left(\rho\left(R\right)-\frac{r}{4}I\right)^{\mp} & W^{-}\left(R\right)+\frac{r}{12}Id
\end{array}\right]
\end{eqnarray*}
 with respect to the adapted frame of \emph{any} orthonormal frame
$\left\{ e_{1},\ldots,e_{4}\right\} $, where $r=2{\rm tr}\left(R\right)$
is the scalar curvature $r={\rm tr}\left(\rho\left(R\right)\right)$.
\end{prop}
\begin{proof}
Since $R$ decomposes as $R=\frac{r}{12}Id+W\left(R\right)+s\left(\rho\left(R\right)-\frac{r}{4}I\right)=\frac{r}{12}Id+W^{+}\left(R\right)+W^{-}\left(R\right)+s\left(\rho\left(R\right)-\frac{r}{4}Id\right)$,
the result follows from the choice of frame for $\Lambda_{2}\left(\mathbb{R}^{4}\right)$.
\end{proof}
In terms of the components $\left\{ R_{ijkl}\right\} $, $W^{+}\left(R\right)+\frac{r}{12}Id$
is given by 
\[
\frac{1}{2}\left[\begin{array}{ccc}
\left(R_{1212}+R_{3434}+2R_{1234}\right) & \left(R_{1213}+R_{3413}\right)-\left(R_{1224}+R_{3424}\right) & \left(R_{1214}+R_{3414}\right)+\left(R_{1223}+R_{3423}\right)\\
\left(R_{1312}-R_{2412}\right)+\left(R_{1334}-R_{2434}\right) & \left(R_{1313}+R_{2424}-2R_{1324}\right) & \left(R_{1314}-R_{2414}\right)+\left(R_{1323}-R_{2423}\right)\\
\left(R_{1412}+R_{2312}\right)+\left(R_{1434}+R_{2334}\right) & \left(R_{1413}+R_{2313}\right)-\left(R_{1424}+R_{2324}\right) & \left(R_{1414}+R_{2323}+2R_{1423}\right)
\end{array}\right],
\]
$W^{-}\left(R\right)+\frac{r}{12}Id$ is given by 

\[
\frac{1}{2}\left[\begin{array}{ccc}
\left(R_{1212}+R_{3434}-2R_{1234}\right) & \left(R_{1213}-R_{3413}\right)+\left(R_{1224}-R_{3424}\right) & \left(R_{1214}-R_{3414}\right)-\left(R_{1223}-R_{3423}\right)\\
\left(R_{1312}+R_{2412}\right)-\left(R_{1334}+R_{2434}\right) & \left(R_{1313}+R_{2424}+2R_{1324}\right) & \left(R_{1314}+R_{2414}\right)-\left(R_{1323}+R_{2423}\right)\\
\left(R_{1412}-R_{2312}\right)-\left(R_{1434}-R_{2334}\right) & \left(R_{1413}-R_{2313}\right)+\left(R_{1424}-R_{2324}\right) & \left(R_{1414}+R_{2323}-2R_{1423}\right)
\end{array}\right],
\]
and $s\left(\rho\left(R\right)-\frac{r}{4}I\right)^{\pm}$ is given
by 
\begin{align*}
 & \frac{1}{2}\left[\begin{array}{ccc}
\left(R_{1212}-R_{3434}\right) & \left(R_{1213}+R_{3413}\right)+\left(R_{1224}+R_{3424}\right) & \left(R_{1214}+R_{3414}\right)-\left(R_{1223}+R_{3423}\right)\\
\left(R_{1312}-R_{2412}\right)-\left(R_{1334}-R_{2434}\right) & \left(R_{1313}-R_{2424}\right) & \left(R_{1314}-R_{2414}\right)-\left(R_{1323}-R_{2423}\right)\\
\left(R_{1412}+R_{2312}\right)-\left(R_{1434}+R_{2334}\right) & \left(R_{1413}+R_{2313}\right)+\left(R_{1424}+R_{2324}\right) & \left(R_{1414}-R_{2323}\right)
\end{array}\right]\\
= & \frac{1}{2}\left[\begin{array}{ccc}
\left(R_{1212}-R_{3434}\right) & \left(\rho\left(R\right)_{23}-\rho\left(R\right)_{14}\right) & \left(\rho\left(R\right)_{24}+\rho\left(R\right)_{13}\right)\\
\left(\rho\left(R\right)_{23}+\rho\left(R\right)_{14}\right) & \left(R_{1313}-R_{2424}\right) & \left(\rho\left(R\right)_{34}-\rho\left(R\right)_{12}\right)\\
\left(\rho\left(R\right)_{24}-\rho\left(R\right)_{13}\right) & \left(\rho\left(R\right)_{34}+\rho\left(R\right)_{12}\right) & \left(R_{1414}-R_{2323}\right)
\end{array}\right]
\end{align*}

\section{Kähler conditions}

If $M$ is a Kähler manifold, the complex-structure tensor $J$ is
an orthogonal transformation on each $T_{*}\left(M,m\right)\cong\mathbb{R}^{2k}$
satisfying $J^{2}=-Id$, so that $J$ is also skew-symmetric. Identifying
$o\left(2k\right)$ with $\Lambda_{2}\left(\mathbb{R}^{2k}\right)$,
$J$, denoted by $I\in\Lambda_{2}\left(\mathbb{R}^{2k}\right)$ to
avoid some confusion, is the metric dual of the Kähler form, $I=\sum_{i<j}a_{ij}e_{i}\wedge e_{j}$
with respect to an arbitrary frame. In dimension 4, the orientation
is consistent with the complex structure if $I\in\Lambda_{2}^{+}\left(\mathbb{R}^{4}\right)$,
so for any oriented frame $\left\{ e_{1},e_{2},e_{3},e_{4}\right\} $,
$I=a_{12}\left(e_{1}\wedge e_{2}+e_{3}\wedge e_{4}\right)+a_{13}\left(e_{1}\wedge e_{3}-e_{2}\wedge e_{4}\right)+a_{14}\left(e_{1}\wedge e_{4}+e_{2}\wedge e_{3}\right)$
with $a_{12}^{2}+a_{13}^{2}+a_{14}^{2}=1$. We will consistently use
this orientation; note that the opposite orientation is used in \cite{LeBrun}.
We can re-order the frame within that orientation so that $a_{12}>0$,
$a_{13}\geq0$ and $a_{14}\geq0$. Unless $a_{12}=1$, we can presume,
again by re-ordering, that $a_{13}>0$ as well. As an operator on
the tangent space, then, 
\begin{eqnarray*}
J\left(e_{1}\right) & = & a_{12}e_{2}+a_{13}e_{3}+a_{14}e_{4}\\
J\left(e_{2}\right) & = & -a_{12}e_{1}+a_{14}e_{3}-a_{13}e_{4}\\
J\left(e_{3}\right) & = & -a_{13}e_{1}-a_{14}e_{2}+a_{12}e_{4}\\
J\left(e_{4}\right) & = & -a_{14}e_{1}+a_{13}e_{2}-a_{12}e_{3}.
\end{eqnarray*}

\subsection{Kähler 4-manifold curvature.}

If $M^{4}$ is Hermitian, the complex structure tensor $J$ can be
extended to an algebraic curvature operator $J:\Lambda_{2}\left(\mathbb{R}^{4}\right)\to\Lambda_{2}\left(\mathbb{R}^{4}\right)$
by $J\left(v\wedge w\right):=J\left(v\right)\wedge J\left(w\right)$.
Like the Hodge star operator $\star$, $J$ is idempotent, $J^{2}=Id$.
For any orthonormal frame $\left\{ e_{1},e_{2},e_{3},e_{4}\right\} $,
\begin{eqnarray*}
J\left(e_{1}\wedge e_{2}-e_{3}\wedge e_{4}\right) & = & J\left(e_{1}\right)\wedge J\left(e_{2}\right)-J\left(e_{3}\right)\wedge J\left(e_{4}\right)\\
 & = & e_{1}\wedge e_{2}-e_{3}\wedge e_{4},
\end{eqnarray*}
etc., so $\left.J\right|_{\Lambda_{2}^{-}\left(\mathbb{R}^{4}\right)}:\Lambda_{2}^{-}\left(\mathbb{R}^{4}\right)\to\Lambda_{2}^{-}\left(\mathbb{R}^{4}\right)$
is the identity on $\Lambda_{2}^{-}\left(\mathbb{R}^{4}\right)$.
Also, $J\left(I\right)=I$. But, using a unitary frame it is trivial
to see that $J:\Lambda_{2}\left(\mathbb{R}^{4}\right)\to\Lambda_{2}\left(\mathbb{R}^{4}\right)$
has a 4-dimensional eigenspace for the eigenvalue $1$, and a 2-dimensional
eigenspace for the eigenvalue $-1$, so necessarily the orthogonal
complement of $I$ within $\Lambda_{2}^{+}\left(\mathbb{R}^{4}\right)$
is that $\left(-1\right)$-eigenspace, spanned by 
\[
{\textstyle \left\{ a_{13}\left(e_{1}\wedge e_{2}+e_{3}\wedge e_{4}\right)-a_{12}\left(e_{1}\wedge e_{3}-e_{2}\wedge e_{4}\right),a_{14}\left(e_{1}\wedge e_{2}+e_{3}\wedge e_{4}\right)-a_{12}\left(e_{1}\wedge e_{4}+e_{2}\wedge e_{3}\right)\right\} }.
\]

If $R$ is a curvature operator corresponding to a Kähler 4-manifold,
then $RJ=JR=R$, so that $\left.R\right|_{\Lambda_{2}^{+}\left(\mathbb{R}^{4}\right)}:\Lambda_{2}^{+}\left(\mathbb{R}^{4}\right)\to\Lambda_{2}\left(\mathbb{R}^{4}\right)$
has at least a two-dimensional kernel, containing
\begin{eqnarray*}
I^{\perp}\cap\Lambda_{2}^{+}\left(\mathbb{R}^{4}\right) & = & \left\{ \left.\xi\right|J\xi=-\xi\right\} \subset\Lambda_{2}^{+}\left(\mathbb{R}^{4}\right).
\end{eqnarray*}
 The conditions for $R$ to be Kähler then become simply
\begin{eqnarray*}
0 & = & R\left(a_{12}\left(e_{1}\wedge e_{3}-e_{2}\wedge e_{4}\right)-a_{13}\left(e_{1}\wedge e_{2}+e_{3}\wedge e_{4}\right)\right),\text{ and}\\
0 & = & R\left(a_{12}\left(e_{1}\wedge e_{4}+e_{2}\wedge e_{3}\right)-a_{14}\left(e_{1}\wedge e_{2}+e_{3}\wedge e_{4}\right)\right),
\end{eqnarray*}
 expanding to
\begin{eqnarray}
0 & = & a_{12}\left(\left(R_{1213}-R_{4243}\right)+\left(R_{2124}-R_{3134}\right)\right)-a_{13}\left(\left(R_{1212}+R_{3434}\right)+2R_{1234}\right)\label{eq:Kaehler Identites}\\
0 & = & a_{12}\left(\left(R_{1214}-R_{3234}\right)-\left(R_{2123}-R_{4143}\right)\right)-a_{14}\left(\left(R_{1212}+R_{3434}\right)+2R_{1234}\right)\nonumber \\
0 & = & a_{13}\left(\left(R_{1213}-R_{4243}\right)+\left(R_{2124}-R_{3134}\right)\right)-a_{12}\left(\left(R_{1313}+R_{2424}\right)-2R_{2413}\right)\nonumber \\
0 & = & a_{13}\left(\left(R_{1314}-R_{2324}\right)-\left(R_{4142}-R_{3132}\right)\right)-a_{14}\left(\left(R_{1313}+R_{2424}\right)-2R_{1324}\right)\nonumber \\
0 & = & a_{14}\left(\left(R_{1214}-R_{3234}\right)-\left(R_{2123}-R_{4143}\right)\right)-a_{12}\left(\left(R_{1414}+R_{2323}\right)+2R_{1423}\right)\nonumber \\
0 & = & a_{14}\left(\left(R_{1314}-R_{2324}\right)+\left(R_{3132}-R_{4142}\right)\right)-a_{13}\left(\left(R_{1414}+R_{2323}\right)+2R_{1423}\right)\nonumber \\
0 & = & a_{12}\left(\rho\left(R\right)_{23}+\rho\left(R\right)_{14}\right)-a_{13}\left(R_{1212}-R_{3434}\right)\nonumber \\
0 & = & a_{12}\left(\rho\left(R\right)_{24}-\rho\left(R\right)_{13}\right)-a_{14}\left(R_{1212}-R_{3434}\right)\nonumber \\
0 & = & a_{13}\left(\rho\left(R\right)_{23}-\rho\left(R\right)_{14}\right)-a_{12}\left(R_{1313}-R_{2424}\right)\nonumber \\
0 & = & a_{13}\left(\rho\left(R\right)_{34}+\rho\left(R\right)_{12}\right)-a_{14}\left(R_{1313}-R_{2424}\right)\nonumber \\
0 & = & a_{14}\left(\rho\left(R\right)_{24}+\rho\left(R\right)_{13}\right)-a_{12}\left(R_{1414}-R_{2323}\right)\nonumber \\
0 & = & a_{14}\left(\rho\left(R\right)_{34}-\rho\left(R\right)_{12}\right)-a_{13}\left(R_{1414}-R_{2323}\right).\nonumber 
\end{eqnarray}

These Kähler identities then imply that 
\begin{eqnarray*}
\frac{1}{a_{12}^{2}}\left(\left(R_{1212}+R_{3434}\right)+2R_{1234}\right) & = & \frac{1}{a_{13}^{2}}\left(\left(R_{1313}+R_{2424}\right)-2R_{2413}\right)=\frac{1}{a_{14}^{2}}\left(\left(R_{1414}+R_{2323}\right)+2R_{1423}\right),
\end{eqnarray*}
thus the scalar curvature satisfies 
\begin{eqnarray}
r & = & \frac{2}{a_{12}^{2}}\left(R_{1212}+R_{3434}+2R_{1234}\right)\label{eq: scalar curvature}\\
 & = & \frac{2}{a_{13}^{2}}\left(R_{1313}+R_{2424}-2R_{1324}\right)\nonumber \\
 & = & \frac{2}{a_{14}^{2}}\left(R_{1414}+R_{2323}+2R_{1423}\right).\nonumber 
\end{eqnarray}
Note that, if a complex-structure component vanishes, say $a_{14}=0$,
then also the associated curvature expression $\left(R_{1414}+R_{2323}+2R_{1423}\right)=0$
as well. If $M^{4}$ is Kähler, then for \emph{any} frame $\left\{ e_{1},e_{2},e_{3},e_{4}\right\} $
in the orientation given by the complex structure $J$, 
\begin{eqnarray*}
R & = & \left[\begin{array}{c|c}
W^{+}\left(R\right)+\frac{r}{12}Id & s\left(\rho\left(R\right)-\frac{r}{4}Id\right)^{\pm}\\
\hline s\left(\rho\left(R\right)-\frac{r}{4}Id\right)^{\mp} & W^{-}\left(R\right)+\frac{r}{12}Id
\end{array}\right]
\end{eqnarray*}
with respect to the adapted frame frame of $\Lambda_{2}\left(\mathbb{R}^{4}\right)$
as before. However, the Kähler conditions imply that the blocks $W^{+}\left(R\right)+\frac{r}{12}Id$,
$s\left(\rho\left(R\right)-\frac{r}{4}Id\right)^{\pm}$, and (the
transpose) $s\left(\rho\left(R\right)-\frac{r}{4}Id\right)^{\mp}$
are all of rank 1. The curvature tensor, using the adapted frame from
Proposition (\ref{=00005BSinger-Thorpe,-AHS=00005D}) becomes 
\begin{prop}
\label{Prop: Kaehler 4-fold curvature tensor}If $M^{4}$ is a Kähler
manifold with an orthonormal frame frame $\left\{ e_{1},e_{2},e_{3},e_{4}\right\} $
with the orientation given by the complex structure $J$, then with
the adapted frame of $\Lambda_{2}\left(\mathbb{R}^{4}\right)$, where
$r=2{\rm tr}\left(R\right)$ is the scalar curvature, the curvature
operator has the form 
\begin{eqnarray*}
R & = & \left[\begin{array}{c|c}
W^{+}\left(R\right)+\frac{r}{12}Id & s\left(\rho\left(R\right)-\frac{r}{4}Id\right)^{\pm}\\
\hline s\left(\rho\left(R\right)-\frac{r}{4}Id\right)^{\mp} & W^{-}\left(R\right)+\frac{r}{12}Id
\end{array}\right]
\end{eqnarray*}
with 
\begin{eqnarray*}
s\left(\rho\left(R\right)-\frac{r}{4}Id\right)^{\pm} & = & \frac{1}{2}\left[\begin{array}{ccc}
a_{12} & 0 & 0\\
a_{13} & 0 & 0\\
a_{14} & 0 & 0
\end{array}\right]\left[\begin{array}{ccc}
\frac{1}{a_{12}}\left(R_{1212}-R_{3434}\right) & \frac{1}{a_{13}}\left(R_{1313}-R_{2424}\right) & \frac{1}{a_{14}}\left(R_{1414}-R_{2323}\right)\\
0 & 0 & 0\\
0 & 0 & 0
\end{array}\right],
\end{eqnarray*}
\begin{eqnarray*}
W^{+}\left(R\right)+\frac{r}{12}Id & = & \frac{r}{4}\left[\begin{array}{ccc}
a_{12} & 0 & 0\\
a_{13} & 0 & 0\\
a_{14} & 0 & 0
\end{array}\right]\left[\begin{array}{ccc}
a_{12} & a_{13} & a_{14}\\
0 & 0 & 0\\
0 & 0 & 0
\end{array}\right],
\end{eqnarray*}
 and 
\begin{eqnarray*}
W^{-}\left(R\right)+\frac{r}{12}Id & = & \frac{r}{4}\left[\begin{array}{ccc}
a_{12} & 0 & 0\\
a_{13} & 0 & 0\\
a_{14} & 0 & 0
\end{array}\right]\left[\begin{array}{ccc}
a_{12} & a_{13} & a_{14}\\
0 & 0 & 0\\
0 & 0 & 0
\end{array}\right]\\
 &  & +\left[\begin{array}{ccc}
-2R_{1234} & -\left(R_{2124}-R_{3134}\right) & \left(R_{2123}-R_{4143}\right)\\
-\left(R_{2124}-R_{3134}\right) & 2R_{2413} & -\left(R_{3132}-R_{4142}\right)\\
\left(R_{2123}-R_{4143}\right) & -\left(R_{3132}-R_{4142}\right) & -2R_{1423}
\end{array}\right].
\end{eqnarray*}
\end{prop}
\begin{proof}
The Kähler identities, equation (\ref{eq:Kaehler Identites}) easily
give these expressions.
\end{proof}
In particular, the two components $W^{+}\left(R\right)$ and $W^{-}\left(R\right)$
of the Weyl tensor, which are the self-dual, and anti-self-dual components,
are given by 
\begin{eqnarray*}
W^{+}\left(R\right) & = & \frac{r}{4}\left[\begin{array}{ccc}
\left(a_{12}^{2}-\frac{1}{3}\right) & a_{12}a_{13} & a_{12}a_{14}\\
a_{12}a_{13} & \left(a_{13}^{2}-\frac{1}{3}\right) & a_{13}a_{14}\\
a_{12}a_{14} & a_{13}a_{14} & \left(a_{14}^{2}-\frac{1}{3}\right)
\end{array}\right]
\end{eqnarray*}
and 
\begin{eqnarray*}
W^{-}\left(R\right) & = & \frac{r}{4}\left[\begin{array}{ccc}
\left(a_{12}^{2}-\frac{1}{3}\right) & a_{12}a_{13} & a_{12}a_{14}\\
a_{12}a_{13} & \left(a_{13}^{2}-\frac{1}{3}\right) & a_{13}a_{14}\\
a_{12}a_{14} & a_{13}a_{14} & \left(a_{14}^{2}-\frac{1}{3}\right)
\end{array}\right]\\
 &  & +\left[\begin{array}{ccc}
-2R_{1234} & -\left(R_{2124}-R_{3134}\right) & \left(R_{2123}-R_{4143}\right)\\
-\left(R_{2124}-R_{3134}\right) & 2R_{2413} & -\left(R_{3132}-R_{4142}\right)\\
\left(R_{2123}-R_{4143}\right) & -\left(R_{3132}-R_{4142}\right) & -2R_{1423}
\end{array}\right]
\end{eqnarray*}
This expression for $W^{+}\left(R\right)$ gives the result of LeBrun
\cite{LeBrun} that any anti-self-dual Kähler 4-manifold (with the
orientation given by the complex structure) must have vanishing scalar
curvature.

\section{Orthogonal coordinates with Kähler 4-manifolds}

Now, what happens when there are orthogonal coordinates? In that case,
as above, necessarily $R_{ijkl}=0$ when all indices are distinct,
that is, $R_{1234}=R_{1324}=R_{1423}=0$, which simplify the Kähler
identities, equations (\ref{eq:Kaehler Identites}), somewhat.

The Kähler condition $\nabla J=0$, or $\nabla_{X}JY=J\nabla_{X}Y$,
gives some specialized formulas as well, which will be used in a special
case. Using the formulas in Proposition (\ref{GM-proposition}) for
the covariant derivatives in orthogonal coordinates,
\begin{eqnarray}
\nabla_{e_{1}}Je_{1} & = & J\nabla_{e_{1}}e_{1}\implies\label{eq:Kaehler conditions}\\
a_{1}e_{1}\left(a_{12}\right) & = & a_{14}e_{3}\left(a_{1}\right)-a_{13}e_{4}\left(a_{1}\right)\nonumber \\
a_{1}e_{1}\left(a_{13}\right) & = & -a_{14}e_{2}\left(a_{1}\right)+a_{12}e_{4}\left(a_{1}\right)\nonumber \\
a_{1}e_{1}\left(a_{14}\right) & = & a_{13}e_{2}\left(a_{1}\right)-a_{12}e_{3}\left(a_{1}\right)\nonumber 
\end{eqnarray}
\begin{eqnarray*}
\nabla_{e_{2}}Je_{1} & = & J\nabla_{e_{2}}e_{1}\implies\\
a_{2}e_{2}\left(a_{12}\right) & = & -a_{13}e_{3}\left(a_{2}\right)-a_{14}e_{4}\left(a_{2}\right)\\
a_{2}e_{2}\left(a_{13}\right) & = & a_{14}e_{1}\left(a_{2}\right)+a_{12}e_{3}\left(a_{2}\right)\\
a_{2}e_{2}\left(a_{14}\right) & = & -a_{13}e_{1}\left(a_{2}\right)+a_{12}e_{4}\left(a_{2}\right)
\end{eqnarray*}
\begin{eqnarray*}
\nabla_{e_{3}}Je_{1} & = & J\nabla_{e_{3}}e_{1}\implies\\
a_{3}e_{3}\left(a_{12}\right) & = & a_{13}e_{2}\left(a_{3}\right)-a_{14}e_{1}\left(a_{3}\right)\\
a_{3}e_{3}\left(a_{13}\right) & = & -a_{12}e_{2}\left(a_{3}\right)-a_{14}e_{4}\left(a_{3}\right)\\
a_{3}e_{3}\left(a_{14}\right) & = & a_{13}e_{4}\left(a_{3}\right)+a_{12}e_{1}\left(a_{3}\right)
\end{eqnarray*}
\begin{eqnarray*}
\nabla_{e_{4}}Je_{1} & = & J\nabla_{e_{4}}e_{1}\implies\\
a_{4}e_{4}\left(a_{12}\right) & = & a_{14}e_{2}\left(a_{4}\right)+a_{13}e_{1}\left(a_{4}\right)\\
a_{4}e_{4}\left(a_{13}\right) & = & a_{14}e_{3}\left(a_{4}\right)-a_{12}e_{1}\left(a_{4}\right)\\
a_{4}e_{4}\left(a_{14}\right) & = & -a_{12}e_{2}\left(a_{4}\right)-a_{13}e_{3}\left(a_{4}\right).
\end{eqnarray*}

\section{Obstructions}

If $M^{2m}$ admits orthogonal coordinates which are unitary, that
is, for some ordering of the coordinates $\left\{ x_{1},\ldots,x_{2m}\right\} $,
the frame $\left\{ e_{1},\ldots,e_{2m}\right\} $ $e_{i}=\frac{1}{a_{i}}\frac{\partial}{\partial x_{i}}$
satisfies $Je_{2k+1}=e_{2k+2}$ and $Je_{2k+2}=-e_{2k+1}$, then the
Kähler conditions will easily show that $M$ must be locally a Riemannian
product of Riemann surfaces.
\begin{prop}
If $M^{2m}$ is a Kähler 2m-manifold supporting orthogonal coordinates
that are unitary, then $M$ is locally a Riemannian product of Riemann
surfaces.
\end{prop}
\begin{proof}
Assume that $M$ has orthogonal coordinates that are unitary as above.
The distribution $\mathcal{D}_{k}:={\rm Span}\left\{ e_{2k+1},e_{2k+2}\right\} $,
which is of course integrable, also satisfies $J\mathcal{D}_{k}=\mathcal{D}_{k}$,
since both $\nabla_{e_{2k+1}}e_{2k+2}=\frac{1}{a_{2k+1}}e_{2k+2}\left(a_{2k+1}\right)e_{2k+1}$
and $\nabla_{e_{2k+2}}e_{2k+1}=\frac{1}{a_{2k+2}}e_{2k+1}\left(a_{2k+2}\right)e_{2k+2}$
are in $\mathcal{D}_{k}$, and
\begin{eqnarray*}
J\left(\nabla_{e_{2k+1}}e_{2k+1}\right) & = & \nabla_{e_{2k+1}}Je_{2k+1}\\
 & = & \nabla_{e_{2k+1}}e_{2k+2}\\
 & \in & \mathcal{D}_{k},
\end{eqnarray*}
so $\nabla_{e_{2k+1}}e_{2k+1}\in\mathcal{D}_{k}$, as is $\nabla_{e_{2k+2}}e_{2k+2}$,
thus $\mathcal{D}_{k}$ is totally geodesic. In addition, 
\begin{eqnarray*}
J\left(\nabla_{e_{2k+1}}e_{2k+1}\right) & = & \nabla_{e_{2k+1}}e_{2k+2}\implies\\
-\sum_{j\neq2k+1}\frac{1}{a_{2k+1}}e_{j}\left(a_{2k+1}\right)Je_{j} & = & \frac{1}{a_{2k+1}}e_{2k+2}\left(a_{2k+1}\right)e_{2k+1}\implies\\
\sum_{j\neq2k+1,2k+2}\frac{1}{a_{2k+1}}e_{j}\left(a_{2k+1}\right)Je_{j} & = & 0,
\end{eqnarray*}
so $a_{2k+1}$depends only on $x_{2k+1}$ and $x_{2k+2}$, and similarly
$a_{2k+2}$depends only on $x_{2k+1}$ and $x_{2k+2}$. Since $\mathcal{D}_{k}\perp\mathcal{D}_{l}$,
the manifold is locally a Riemannian product of Riemann surfaces. 
\end{proof}
Assume in the following that that $M$ is a 4-real-dimensional Kähler
manifold, which supports orthogonal coordinates. Then, with respect
to an oriented frame $\left\{ e_{1},e_{2},e_{3},e_{4}\right\} $ associated
to an orthogonal coordinate chart, the following restrictions apply.
\begin{prop}
If $M$ is a 4-dimensional Kähler manifold supporting orthogonal coordinates,
then, for any frame $\left\{ e_{1},\ldots,e_{4}\right\} $ arising
from orthogonal coordinates, the scalar curvature $r$ satisfies
\begin{eqnarray*}
r & = & \frac{2}{a_{12}^{2}}\left(R_{1212}+R_{3434}\right)=\frac{2}{a_{13}^{2}}\left(R_{1313}+R_{2424}\right)=\frac{2}{a_{14}^{2}}\left(R_{1414}+R_{2323}\right).
\end{eqnarray*}
$\left(R_{1212}+R_{3434}\right),\,\left(R_{1313}+R_{2424}\right),\text{ and }\left(R_{1414}+R_{2323}\right)$
thus all have the same sign. 
\end{prop}
\begin{proof}
The second line follows from equation (\ref{eq: scalar curvature}),
and the fact that $R_{1234}=R_{1324}=R_{1423}=0$. 
\end{proof}
The next result follows immediately from {[}\ref{Prop: Kaehler 4-fold curvature tensor}{]}
and the fact that $R_{1234}=R_{1324}=R_{1423}=0$. 
\begin{prop}
If $M^{4}$ is a Kähler manifold with orthogonal coordinates, then,
for a frame $\left\{ e_{1},e_{2},e_{3},e_{4}\right\} $ associated
to an orthogonal coordinate chart, the curvature is of the form 
\begin{eqnarray*}
R & = & \left[\begin{array}{c|c}
W^{+}\left(R\right)+\frac{r}{12}Id & s\left(\rho\left(R\right)-\frac{r}{4}I\right)^{\pm}\\
\hline s\left(\rho\left(R\right)-\frac{r}{4}I\right)^{\mp} & W^{-}\left(R\right)+\frac{r}{12}Id
\end{array}\right]
\end{eqnarray*}
with 
\begin{eqnarray*}
W^{-}\left(R\right) & = & \frac{r}{4}\left[\begin{array}{ccc}
\left(a_{12}^{2}-\frac{1}{3}\right) & a_{12}a_{13} & a_{12}a_{14}\\
a_{12}a_{13} & \left(a_{13}^{2}-\frac{1}{3}\right) & a_{13}a_{14}\\
a_{12}a_{14} & a_{13}a_{14} & \left(a_{14}^{2}-\frac{1}{3}\right)
\end{array}\right]\\
 &  & +\left[\begin{array}{ccc}
0 & -\left(R_{2124}-R_{3134}\right) & \left(R_{2123}-R_{4143}\right)\\
-\left(R_{2124}-R_{3134}\right) & 0 & -\left(R_{3132}-R_{4142}\right)\\
\left(R_{2123}-R_{4143}\right) & -\left(R_{3132}-R_{4142}\right) & 0
\end{array}\right].
\end{eqnarray*}
\end{prop}

\subsection{Self-Dual Kähler 4-manifolds.}
\begin{thm}
If $M^{4}$ is a self-dual Kähler 4-manifold, then it cannot support
orthogonal coordinates unless it is flat, or is a product of two Riemann
surfaces of opposite constant curvature. 
\end{thm}
\begin{rem}
See \cite{Derd} for a discussion of self-dual Kähler 4-manifolds.
\end{rem}
\begin{proof}
Assuming that $M$ is self-dual, with respect to a frame $\left\{ e_{1},e_{2},e_{3},e_{4}\right\} $
from an orthogonal coordinate chart,
\begin{eqnarray*}
0 & = & W^{-}\left(R\right)\\
 & = & \left[\begin{array}{ccc}
\frac{r}{4}\left(a_{12}^{2}-\frac{1}{3}\right) & -\left(R_{2124}-R_{3134}\right)+\frac{r}{4}a_{12}a_{13} & \left(R_{2123}-R_{4143}\right)+\frac{r}{4}a_{12}a_{14}\\
-\left(R_{2124}-R_{3134}\right)+\frac{r}{4}a_{12}a_{13} & \frac{r}{4}\left(a_{13}^{2}-\frac{1}{3}\right) & -\left(R_{3132}-R_{4142}\right)+\frac{r}{4}a_{13}a_{14}\\
\left(R_{2123}-R_{4143}\right)+\frac{r}{4}a_{12}a_{14} & -\left(R_{3132}-R_{4142}\right)+\frac{r}{4}a_{13}a_{14} & \frac{r}{4}\left(a_{14}^{2}-\frac{1}{3}\right)
\end{array}\right]
\end{eqnarray*}
so that, either $r=0$, or $a_{12}^{2}=a_{13}^{2}=a_{14}^{2}=\frac{1}{3}$.
In the first case, then also $W^{+}\left(R\right)=0$, since $r=0$
and $W^{+}\left(R\right)=\frac{r}{4}\Pi_{I}-\frac{r}{12}Id$. Thus,
the manifold must be conformally flat, thus by \cite{Tanno} must
be flat or a product of two Riemann surfaces with opposite constant
Gaussian curvatures. 

In the second case, then applying equations (\ref{eq:Kaehler conditions}),
in the special case $a_{12}=a_{13}=a_{14}=\frac{1}{\sqrt{3}}$. 
\begin{align}
c_{1} & := & \frac{1}{a_{1}}e_{2}\left(a_{1}\right) & = & \frac{1}{a_{1}}e_{3}\left(a_{1}\right) & = & \frac{1}{a_{1}}e_{4}\left(a_{1}\right)\nonumber \\
c_{2} & := & \frac{1}{a_{2}}e_{1}\left(a_{2}\right) & = & -\frac{1}{a_{2}}e_{3}\left(a_{2}\right) & = & \frac{1}{a_{2}}e_{4}\left(a_{2}\right)\nonumber \\
c_{3} & := & \frac{1}{a_{3}}e_{1}\left(a_{3}\right) & = & \frac{1}{a_{3}}e_{2}\left(a_{3}\right) & = & -\frac{1}{a_{3}}e_{4}\left(a_{3}\right)\nonumber \\
c_{4} & := & \frac{1}{a_{4}}e_{1}\left(a_{4}\right) & = & -\frac{1}{a_{4}}e_{2}\left(a_{4}\right) & = & \frac{1}{a_{4}}e_{3}\left(a_{4}\right)\label{eq: relations between derivatives of a_i in special case}
\end{align}
Add to that, following the clever argument of {[}\cite{GM}, p. 6{]},
that, since both $\left(e_{2}-e_{3}\right)\left(a_{1}\right)=0$ and
$\left(e_{2}-e_{4}\right)\left(a_{1}\right)=0$, that also $\left[e_{2}-e_{3},e_{2}-e_{4}\right]\left(a_{1}\right)=0$,
and since 
\begin{eqnarray*}
0 & = & \left[e_{2}-e_{3},e_{2}-e_{4}\right]\left(a_{1}\right)\\
 & = & -\left[e_{2},e_{4}\right]\left(a_{1}\right)+\left[e_{2},e_{3}\right]\left(a_{1}\right)+\left[e_{3},e_{4}\right]\left(a_{1}\right)\\
 & = & \frac{1}{a_{2}}\left(e_{3}-e_{4}\right)\left(a_{2}\right)e_{2}\left(a_{1}\right)+\frac{1}{a_{4}}\left(e_{2}-e_{3}\right)\left(a_{4}\right)e_{4}\left(a_{1}\right)+\frac{1}{a_{3}}\left(e_{4}-e_{2}\right)\left(a_{3}\right)e_{3}\left(a_{1}\right)\\
 & = & \frac{2}{a_{2}}e_{3}\left(a_{2}\right)e_{2}\left(a_{1}\right)+\frac{2}{a_{3}}e_{4}\left(a_{3}\right)e_{3}\left(a_{1}\right)+\frac{2}{a_{4}}e_{2}\left(a_{4}\right)e_{4}\left(a_{1}\right),
\end{eqnarray*}
then either $\left(\frac{2}{a_{2}}e_{3}\left(a_{2}\right)+\frac{2}{a_{3}}e_{4}\left(a_{3}\right)+\frac{2}{a_{4}}e_{2}\left(a_{4}\right)\right)=0$
or $e_{2}\left(a_{1}\right)=0$, or, in terms of $c_{j}$, either
$\left(c_{2}+c_{3}+c_{4}\right)=0$ or $c_{1}=0$.

Similarly, the second line in equations (\ref{eq: relations between derivatives of a_i in special case})
implies that $\left(e_{1}+e_{3}\right)\left(a_{2}\right)=0$ and $\left(e_{1}-e_{4}\right)\left(a_{2}\right)=0$,
so 
\begin{eqnarray*}
0 & = & \left[e_{1}+e_{3},e_{1}-e_{4}\right]\left(a_{2}\right)\\
 & = & \left(-\left[e_{1},e_{4}\right]-\left[e_{1},e_{3}\right]-\left[e_{3},e_{4}\right]\right)\left(a_{2}\right)\\
 & = & \left(-2\frac{1}{a_{1}}e_{3}\left(a_{1}\right)-2\frac{1}{a_{3}}e_{1}\left(a_{3}\right)+2\frac{1}{a_{4}}e_{1}\left(a_{4}\right)\right)e_{1}\left(a_{2}\right),
\end{eqnarray*}
so, either $\left(-2\frac{1}{a_{1}}e_{3}\left(a_{1}\right)-2\frac{1}{a_{3}}e_{1}\left(a_{3}\right)+2\frac{1}{a_{4}}e_{1}\left(a_{4}\right)\right)=0$
or $e_{1}\left(a_{2}\right)=0$; in terms of the $c_{j}$, either
$\left(c_{1}+c_{3}-c_{4}\right)=0$ or $c_{2}=0$. 

From the third line in equations (\ref{eq: relations between derivatives of a_i in special case}),
we have that $\left(e_{1}-e_{2}\right)\left(a_{3}\right)=0=\left(e_{1}+e_{4}\right)\left(a_{3}\right)$,
\begin{eqnarray*}
0 & = & \left[e_{1}-e_{2},e_{1}+e_{4}\right]\left(a_{3}\right)\\
 & = & \left[e_{1},e_{4}\right]\left(a_{3}\right)+\left[e_{1},e_{2}\right]\left(a_{3}\right)-\left[e_{2},e_{4}\right]\left(a_{3}\right)\\
 & = & \left(2\frac{1}{a_{1}}e_{2}\left(a_{1}\right)-2\frac{1}{a_{2}}e_{1}\left(a_{2}\right)+2\frac{1}{a_{4}}e_{1}\left(a_{4}\right)\right)e_{1}\left(a_{3}\right),
\end{eqnarray*}
so $c_{1}-c_{2}+c_{4}$ or $c_{3}=0$. Finally, since $\left(e_{1}+e_{2}\right)\left(a_{4}\right)=0=\left(e_{1}-e_{3}\right)\left(a_{4}\right)$,
\begin{eqnarray*}
0 & = & \left[e_{1}+e_{2},e_{1}-e_{3}\right]\left(a_{4}\right)\\
 & = & -\left[e_{1},e_{3}\right]\left(a_{4}\right)-\left[e_{1},e_{2}\right]\left(a_{4}\right)-\left[e_{2},e_{3}\right]\left(a_{4}\right)\\
 & = & \left(-2\frac{1}{a_{1}}e_{2}\left(a_{1}\right)-2\frac{1}{a_{2}}e_{1}\left(a_{2}\right)+2\frac{1}{a_{3}}e_{1}\left(a_{3}\right)\right)e_{1}\left(a_{4}\right).\\
\end{eqnarray*}
so $\left(c_{1}+c_{2}-c_{3}\right)=0$ or $c_{4}=0$. But then, either
\begin{eqnarray*}
\left[\begin{array}{cccc}
0 & 1 & 1 & 1\\
1 & 0 & 1 & -1\\
1 & -1 & 0 & 1\\
1 & 1 & -1 & 0
\end{array}\right]\left[\begin{array}{c}
c_{1}\\
c_{2}\\
c_{3}\\
c_{4}
\end{array}\right] & = & \left[\begin{array}{c}
0\\
0\\
0\\
0
\end{array}\right]
\end{eqnarray*}
or some of the $c_{i}=0$. If one entry vanishes, say $c_{1}=0$,
then unless another vanishes, 
\begin{eqnarray*}
\left[\begin{array}{ccc}
0 & 1 & -1\\
-1 & 0 & 1\\
1 & -1 & 0
\end{array}\right]\left[\begin{array}{c}
c_{2}\\
c_{3}\\
c_{4}
\end{array}\right] & = & \left[\begin{array}{c}
0\\
0\\
0
\end{array}\right],
\end{eqnarray*}
which again has maximal rank, so another entry must vanish, say $c_{2}$.
Continuing, necessarily all $c_{j}$ must vanish.
\end{proof}
This argument is the only occurrence, in this section, where we use
more than the condition that $R_{ijkl}=0$ for $i,j,k,l$ distinct,
in the obstructions to the existence of orthogonal coordinates. As
mentioned in \cite{GM}, embedded in the proof is the fact that the
algebraic condition $R_{ijkl}=0$ for $i,j,k,l$ distinct, for \emph{some}
frame, does not imply the existence of orthogonal coordinates, in
that:
\begin{example}
There is a frame $\left\{ e_{1},e_{2},e_{3},e_{4}\right\} $ of $\mathbb{CP}^{2}$
with the Fubini-Study metric for which $R_{ijkl}=0$ whenever $i,j,k,l$
are distinct.
\end{example}
\begin{proof}
Start with any unitary frame $\left\{ u_{1},u_{2},u_{3},u_{4}\right\} $
on $\mathbb{CP}^{2}$. Then set 
\begin{eqnarray*}
e_{1} & = & u_{1}\\
e_{2} & = & \frac{1}{\sqrt{3}}u_{2}+\frac{1}{\sqrt{2}}u_{3}+\frac{1}{\sqrt{6}}u_{4}\\
e_{3} & = & \frac{1}{\sqrt{3}}u_{2}-\frac{1}{\sqrt{2}}u_{3}+\frac{1}{\sqrt{6}}u_{4}\\
e_{4} & = & \frac{1}{\sqrt{3}}u_{2}-\frac{\sqrt{2}}{\sqrt{3}}u_{4}.
\end{eqnarray*}
 This clearly gives a frame $\left\{ e_{1},e_{2},e_{3},e_{4}\right\} $
on $\mathbb{CP}^{2}$, which satisfies the conditions $a_{12}=a_{13}=a_{14}=\frac{1}{\sqrt{3}}$
for the complex structure tensor with respect to that frame, and $R_{ijkl}=0$
whenever all indices are distinct.
\end{proof}

\subsection{Ricci flat}

A compact (real) 4-dimensional Kähler manifold that is Ricci-flat
must be a K3 surface, thus is a projective manifold \cite{Tosatti}.
Our final result shows that no nontrivial Kähler 4-manifold which
is Ricci-flat can support orthogonal coordinates. 
\begin{thm}
If $M^{4}$ is Kähler and Ricci-flat, it does not support orthogonal
coordinates unless it is flat.
\end{thm}
\begin{proof}
If the manifold supports orthogonal coordinates and is Ricci-flat,
then, with respected to an associated frame as above, noting that
Ricci flatness implies that $r=0$ and $\left(R_{2124}-R_{3134}\right)=2R_{2124}$,
etc., 
\begin{eqnarray*}
R & = & \left[\begin{array}{cc}
\left[\begin{array}{ccc}
0 & 0 & 0\\
0 & 0 & 0\\
0 & 0 & 0
\end{array}\right] & \left[\begin{array}{ccc}
0 & 0 & 0\\
0 & 0 & 0\\
0 & 0 & 0
\end{array}\right]\\
\left[\begin{array}{ccc}
0 & 0 & 0\\
0 & 0 & 0\\
0 & 0 & 0
\end{array}\right] & \left[\begin{array}{ccc}
0 & -2R_{2124} & 2R_{2123}\\
-2R_{2124} & 0 & -2R_{3132}\\
2R_{2123} & -2R_{3132} & 0
\end{array}\right]
\end{array}\right].
\end{eqnarray*}
Since $\rho\left(R\right)=0$, $W^{+}\left(R\right)+\frac{r}{12}Id=0,$
and $W^{+}\left(R\right)+\frac{r}{12}Id=\left[\begin{array}{ccc}
0 & -2R_{2124} & 2R_{2123}\\
-2R_{2124} & 0 & -2R_{3132}\\
2R_{2123} & -2R_{3132} & 0
\end{array}\right]$,
\begin{eqnarray*}
0 & = & \left<R\left(e_{1}\wedge e_{2}+e_{3}\wedge e_{4}\right),e_{1}\wedge e_{4}+e_{2}\wedge e_{3}\right>\\
 & = & R_{1214}+R_{1223}+R_{3414}+R_{3423}\\
 & = & 2\left(R_{1214}+R_{1223}\right)
\end{eqnarray*}
\begin{eqnarray*}
0 & = & \left<R\left(e_{1}\wedge e_{2}+e_{3}\wedge e_{4}\right),e_{1}\wedge e_{3}-e_{2}\wedge e_{4}\right>\\
 & = & R_{1213}-R_{1224}+R_{3413}-R_{3424}\\
 & = & 2\left(R_{1213}-R_{1224}\right)
\end{eqnarray*}
\begin{eqnarray*}
0 & = & \left<R\left(e_{1}\wedge e_{3}-e_{2}\wedge e_{4}\right),e_{1}\wedge e_{4}+e_{2}\wedge e_{3}\right>\\
 & = & R_{1314}+R_{1323}-R_{2414}-R_{2423}\\
 & = & 2\left(R_{1314}+R_{1323}\right)
\end{eqnarray*}
and
\begin{eqnarray*}
-2R_{2124} & = & \left<R\left(e_{1}\wedge e_{2}-e_{3}\wedge e_{4}\right),e_{1}\wedge e_{3}+e_{2}\wedge e_{4}\right>\\
2R_{1224} & = & R_{1213}+R_{1224}-R_{3413}-R_{3424}\\
0 & = & R_{1213}-R_{1224}-R_{3413}-R_{3424}\\
0 & = & 2R_{1213}
\end{eqnarray*}
\begin{eqnarray*}
2R_{2123} & = & \left<R\left(e_{1}\wedge e_{2}-e_{3}\wedge e_{4}\right),e_{1}\wedge e_{4}-e_{2}\wedge e_{3}\right>\\
-2R_{1223} & = & R_{1214}-R_{1223}-R_{3414}+R_{3423}\\
0 & = & R_{1214}+R_{1223}-R_{3414}+R_{3423}\\
0 & = & 2R_{1214}
\end{eqnarray*}
\begin{eqnarray*}
-2R_{3132} & = & \left<R\left(e_{1}\wedge e_{3}+e_{2}\wedge e_{4}\right),e_{1}\wedge e_{4}-e_{2}\wedge e_{3}\right>\\
-2R_{1323} & = & R_{1314}-R_{1323}+R_{2414}-R_{2423}\\
0 & = & R_{1314}+R_{1323}+R_{2414}-R_{2423}\\
0 & = & 2R_{1314},
\end{eqnarray*}
so 
\[
0=R_{1213}=R_{1214}=R_{1314}=R_{1323}=R_{1224}=R_{1223},
\]
and the manifold is flat.
\end{proof}

\end{document}